\documentclass[a4paper,12pt]{amsart}

\usepackage{graphicx,graphics}
\usepackage{latexsym,amssymb,amsfonts,amsmath}
\usepackage[mathcal]{euscript}

\newcommand{\bR}{\mathbb R}
\newcommand{\bC}{\mathbb C}

\newcommand{\bN}{\mathbb N}
\renewcommand{\Re}{\mathop\mathrm{Re}\nolimits}
\renewcommand{\Im}{\mathop\mathrm{Im}\nolimits}
\newcommand{\Int}[1]{\mathop\mathrm{Int}\nolimits #1}
\renewcommand{\emptyset}{\varnothing}

\newtheorem{theorem}{Theorem}
\newtheorem{lemma}{Lemma}
\newtheorem{prop}{Proposition}
\newtheorem{cor}{Corollary}
\newtheorem{defn}{Definition}

\title{On conjugate pseudo-harmonic functions.}
\author{Yevgen Polulyakh}
\thanks{Institute of mathematics of NAS of Ukraine, \textsl{e-mail: polulyah@imath.kiev.ua}}

\begin{document}
\maketitle

\begin{abstract}
We prove the following theorem. Let $U$ be a pseudo-harmonic function on a surface $M^2$. For a real valued continuous function $V : M^2 \rightarrow \bR$ to be a conjugate pseudo-harmonic function of $U$ on $M^2$ it is necessary and sufficient that $V$ is open on level sets of $U$.
\end{abstract}

\bigskip

\begin{center}
\parbox{0.8\textwidth}%
{
	\begin{footnotesize}
	\end{footnotesize}
}
\end{center}

\vspace{3ex}

Let $M^2$ be a surface, i.e. a 2-dimensional and separable manifold, $U : M^2 \rightarrow \bR$ be a real-valued function on $M^2$. Denote also by
$$
D = \{ (x, y) \in \bR^2 \,|\, x^2 + y^2 < 1 \}
$$
the open unit disk in the plane.

\begin{defn}[see~\cite{Toki,Morse}]\label{defn_pseudo_harm_loc}
A function $U$ is called \emph{pseudo-harmonic in a point $p \in M^2$} if there exist a neighbourhood $N$ of $p$ on $M^2$ and a homeomorphism $T : D \rightarrow N$ such that $T(0, 0) = p$ and a function
$$
u = U \circ T : D \rightarrow \bR^2
$$
is harmonic and not identically constant.

A neighbourhood $N$ is called \emph{simple neighbourhood of $p$}.
\end{defn}

We can even choose $N$ and $T$ from previous definition to comply with the equality
$$
u(z) = U \circ T(z) = \Re z^n + U(p) \,, \quad z = x + iy \in D \,,
$$
for a certain $n = n(p) \in \bN$ (see~\cite{Morse}).

\begin{defn}[see~\cite{Toki,Morse}]\label{defn_pseudo_harm_glob}
A function $U$ is called \emph{pseudo-harmonic on $M^2$} if it is pseudo-harmonic in each point $p \in M^2$.
\end{defn}

Let $U : M^2 \rightarrow \bR$ be a pseudo-harmonic function on $M^2$ and $V : M^2 \rightarrow \bR$ be a real valued function.

\begin{defn}[see~\cite{Toki}]\label{defn_conjugate_loc}
A function $V$ is called a \emph{conjugate pseudo-harmonic function of $U$ in a point $p \in M^2$} if there exist a neighbourhood $N$ of $p$ on $M^2$ and a homeomorphism $T : D \rightarrow N$ such that $T(0, 0) = p$ and
$$
u = U \circ T : D \rightarrow \bR^2
\quad \mbox{and} \quad
v = V \circ T : D \rightarrow \bR^2
$$
are conjugate harmonic functions.
\end{defn}

We can choose $N$ and $T$ from previous definition in such way that
\begin{align*}
u(z) & = U \circ T(z) = \Re z^n + U(p) \,, \\
v(z) & = V \circ T(z) = \Im z^n + V(p) \,, \quad z = x + iy \in D \,,
\end{align*}
for a certain $n = n(p) \in \bN$ (see~\cite{Morse}).

\begin{defn}[see~\cite{Toki}]\label{defn_conjugate_glob}
A function $V$ is called a \emph{conjugate pseudo-harmonic function of $U$ on $M^2$} if it is a conjugate pseudo-harmonic function of $U$ in every $p \in M^2$.
\end{defn}

\begin{defn}\label{defn_open_on_levelsets}
Let $U$ and $V$ be continuous real valued functions on a surface $M^2$. We say that \emph{$V$ is open on level sets of $U$} if for every $c \in U(M^2)$ a mapping
$$
V |_{U^{-1}(c)} : U^{-1}(c) \rightarrow \bR
$$
is open on the space $U^{-1}(c)$ in the topology induced from $M^2$.
\end{defn}

\begin{theorem}\label{th_conjugace_conditions}
Let $U$ be a pseudo-harmonic function on $M^2$. For a real valued continuous function $V : M^2 \rightarrow \bR$ to be a conjugate pseudo-harmonic function of $U$ on $M^2$ it is necessary and sufficient that $V$ is open on level sets of $U$.
\end{theorem}

Let us remind following
\begin{defn}[see~\cite{Stoilov}]\label{defn_interior_mapping}
A mapping $G : M^2_1 \rightarrow M^2_2$ of a surface $M^2_1$ to a surface $M^2_2$ is called \emph{interior} if it complies with conditions:
\begin{itemize}
	\item[1)] $G$ is open, i. e. an image of any open subset of $M^2_1$ is open in $M^2_2$;
	\item[2)] for every $p \in M^2_2$ its full preimage $G^{-1}(p)$ does not contain any nondegenerate continuum (closed connected subset of $M^2_1$).
\end{itemize}

\end{defn}

In order to prove theorem~\ref{th_conjugace_conditions} we need following

\begin{lemma}\label{lemma_interior_mapping}
Let $U$ be a pseudo-harmonic function on $M^2$. Let a real valued continuous function $V$ be open on level sets of $U$.

Then the mapping $F : M^2 \rightarrow \bC$,
$$
F(p) = U(p) + i V(p) \,, \quad p \in M^2
$$
is interior.
\end{lemma}

First we will verify one auxiliary statement. Denote $I = [0,1]$, $\mathring{I} = (0,1) = I \setminus \{0,1\}$.

\begin{prop}\label{prop_monotone}
In the condition of Lemma~\ref{lemma_interior_mapping} the following statement holds true.

Let $\gamma : I \rightarrow M^2$ be a simple continuous curve and $\gamma(I) \subseteq U^{-1}(c)$ for a certain $c \in \bR$. If the set $\gamma(\mathring{I})$ is open in $U^{-1}(c)$ in the topology induced from $M^2$, then the function $V \circ \gamma : I \rightarrow \bR$ is strictly monotone.
\end{prop}

\begin{proof}
Suppose that contrary to the statement of Proposition the equality $V \circ \gamma(\tau_1) = V \circ \gamma(\tau_2)$ is valid for certain $\tau_1, \tau_2 \in I$, $\tau_1 < \tau_2$.

Since the function $V \circ \gamma$ is continuous and a set $[\tau_1, \tau_2]$ is compact, then following values
\begin{align*}
d_1 & = \min_{t \in [\tau_1, \tau_2]} V \circ \gamma(t) \,,\\
d_2 & = \max_{t \in [\tau_1, \tau_2]} V \circ \gamma(t) \,,
\end{align*}
are well defined. Let us fix $s_1, s_2 \in [\tau_1, \tau_2]$ such that $d_i = V \circ \gamma(s_i)$, $i = 1,2$.

We designate $W = (\tau_1, \tau_2)$. It is obviously the open subset of $\mathring{I}$.

Let us consider first the case $d_1 = d_2$. It is clear that $[\tau_1, \tau_2] \subseteq (V \circ \gamma)^{-1}(d_1)$ in this case. So the open subset $\gamma(W)$ of the level set $U^{-1}(c)$ is mapped by $V$ onto a one-point set $\{d_1\}$ which is not open in $\bR$ and $V$ is not open on level sets of $U$.

Assume now that $d_1 \neq d_2$. Since $V \circ \gamma(\tau_1) = V \circ \gamma(\tau_2)$ due to our previous supposition, then either $s_1$ or $s_2$ is contained in $W$.

Let $s_1 \in W$ (the case $s_2 \in W$ is considered similarly). Then $V \circ \gamma(W) \subseteq [d_1, +\infty)$ and the open subset $\gamma(W)$ of the level set $U^{-1}(c)$ can not be mapped by $V$ to an open subset of $\bR$ since its image containes the frontier point $d_1 = V \circ \gamma(s_1)$. So, in this case $V$ is not open on level sets of $U$.

The contradiction obtained shows that our initial supposition is false and the function $V \circ \gamma$ is strictly monotone on $I$.
\end{proof}

\begin{proof}[Proof of Lemma~\ref{lemma_interior_mapping}]
Let $p \in M^2$ and $Q$ be an open neighbourhood of $p$.

We are going to show that the set $F(Q)$ containes a neigbourhood of $F(p)$. At the same time we shall show that $p$ is an isolated point of a level set $F^{-1}(F(p))$.

Without loss of generality we can assume that $U(p) = V(p) = 0$.

Let $N$ be a simple neighbourhood of $p$ and $T : D \rightarrow N$ be a homeomorphism such that for a certain $n \in \bN$ the folloving equality holds true $u(z) = U \circ T(z) = \Re z^n$, $z \in D$ (see Definition~\ref{defn_pseudo_harm_loc} and the subsequent remark). It is clear that without losing generality we can regard that $N$ is small enough to be contained in $Q$.

Observe that for an arbitrary level set $\Gamma$ of $U$ an intersection $\Gamma \cap T(D) = \Gamma \cap N$ is open in $\Gamma$. Consequently, since $T$ is homeomorphism then a mapping $v = V \circ T : D \rightarrow \bR$ is open on level sets of $u = U \circ T : D \rightarrow \bR$ (see Definition~\ref{defn_open_on_levelsets}).

Let us consider two possibilities.

\medskip

{\bf Case 1.} Zero is a regular point of the smooth function $u = U \circ T$, i. e. $n = 1$ and $u(z) = \Re z$, $z \in D$.

In this case $u^{-1}(u(0)) = u^{-1}(U(p)) = T^{-1}(U^{-1}(U(p))) = \{0\} \times (-1, 1)$. According to Proposition~\ref{prop_monotone} the function $v$ is strictly monotone on every segment which is contained in this interval, so it is strictly monotone on $\{0\} \times (-1, 1)$. Consequently, for points $z_1 = 0-i/2$ and $z_2 = 0+i/2$ the following inequality holds true $v(z_1) \cdot v(z_2) < 0$.

Let us note that from previous it follows that $V$ is monotone on the arc $\beta = T(\{0\} \times (-1, 1)) = U^{-1}(U(p)) \cap N$. And since $F^{-1}(F(p)) \cap N \subset \beta$ then $F^{-1}(F(p)) \cap N = \{p\}$ and $p$ is an isolated point of its level set $F^{-1}(F(p))$.

Let $d_1 = v(z_1) < 0$ and $d_2 = v(z_2) > 0$ (The case $d_1 > 0$ and $d_2 < 0$ is considered similarly). Denote
$$
\varepsilon = \frac{1}{2} \min (|d_1|, |d_2|) > 0 \,.
$$
Function $v$ is continuous, so there exists $\delta > 0$ such that following implications are fulfilled
\begin{eqnarray*}
|z-z_1| < \delta & \Rightarrow & |v(z) - d_1| < \varepsilon \,,\\
|z-z_2| < \delta & \Rightarrow & |v(z) - d_2| < \varepsilon \,.
\end{eqnarray*}

\begin{figure}[htbp]
\begin{center}
\includegraphics{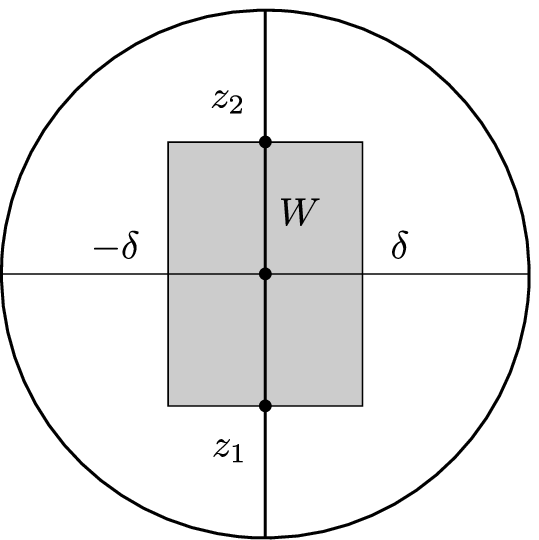}
\caption{}\label{fig_regular}
\end{center}
\end{figure}

Let us examine a neighbourhood $W = (-\delta, \delta) \times (-1/2, 1/2)$ of 0, which is depicted on Figure~\ref{fig_regular}. It can be easily seen that for every $x \in (-\delta, \delta)$ following relations are valid
\begin{gather*}
u(x+iy) = x \,, \quad y \in (-\varepsilon, \varepsilon) \,,\\
v(x-i/2) < v(z_1) + \varepsilon < -2\varepsilon + \varepsilon = -\varepsilon \,,\\
v(x+i/2) > v(z_2) - \varepsilon > 2\varepsilon - \varepsilon = \varepsilon \,.
\end{gather*}
From two last lines and from the continuity of $v$ on a segment $\{x\} \times [-1/2, 1/2]$ it follows that $v(\{x\} \times [-1/2, 1/2]) \supseteq (-\varepsilon, \varepsilon)$. Therefore
$$
F \circ T (\{x\} \times [-1/2, 1/2]) \supseteq \{x\} \times (-\varepsilon, \varepsilon) \,, \quad x \in (-\delta, \delta) \,.
$$

Since $T(W) \subseteq N \subseteq Q$ by the choise of $N$, then
$$
0 = F(p) \in (-\delta, \delta) \times (-\varepsilon, \varepsilon) \subseteq F \circ T (W) \subseteq F(Q) \,.
$$

\medskip

{\bf Case 2.} Zero is a saddle point of $u = U \circ T$, i. e. $u(z) = \Re z^n$, $z \in D$ for a certain $n > 1$.

In this case
$$
u^{-1}(u(0)) = T^{-1}(U^{-1}(U(p))) = \{0\} \cup \bigcup_{k=0}^{2n-1} \gamma_k \,,
$$
where $\gamma_k = \{ z \in D \;|\; z = a \cdot \exp(\pi i (k-1/2)/n) ,\; a \in (0,1) \}$, $k = 1, \ldots, 2n-1$.

As above, applying Proposition~\ref{prop_monotone} we conclude that function $v = V \circ T$ is strictly monotone on each arc $\gamma_k$, $k = 1, \ldots, 2n-1$. Since $v$ is continuous and $0$ is a boundary point for each $\gamma_k$, then $v(z) \neq v(0)$ for all $z \in \bigcup_{k} \gamma_k$. Therefore, $0 = (F \circ T)^{-1}(F \circ T(0))$ and $F^{-1}(F(p)) \cap N = \{p\}$, i. e. $p$ is the isolated point if its level set $F^{-1}(F(p)$.

Let us designate by
\begin{multline*}
\textstyle
R_k = \Bigl\{ z \in D \bigm| z = a e^{i \varphi},\; a \in [0, 1),\; \varphi \in \Bigl[ \frac{\pi(k-1/2)}{2}, \frac{\pi(k+1/2)}{2} \Bigr] \Bigr\} \,,\\
k = 0, \ldots, 2n-1
\end{multline*}
sectors on which disk $D$ is divided by the level set $u^{-1}(u(0))$.

We also denote
\begin{align*}
D_l & = \{ z \in D \;|\; \Re z \leq 0 \} \,,\\
D_r & = \{ z \in D \;|\; \Re z \geq 0 \} \,.
\end{align*}

Consider map $\Phi : D \rightarrow D$ given by the formula $\Phi(z) = z^n$, $z \in D$. It is easy to see that for every $k \in \{0, \ldots, 2n-1\}$ depending on its parity sector $R_k$ is mapped homeomorphically by $\Phi$ either onto $D_l$ or onto $D_r$. Let a mapping $\Phi_k : R_k \rightarrow D_r$ is given by relation
$$
\Phi_k =
\begin{cases}
	\Phi |_{R_k} , & \mbox{if } k = 2m \,,\\
	\mathrm{Inv} \circ \Phi |_{R_k} , & \mbox{if } k = 2m+1 \,,\\
\end{cases}
\quad k = 0, \ldots, 2n-1 \,,
$$
where $\mathrm{Inv} : D \rightarrow D$ is defined by formula $\mathrm{Inv}(z) = -z$, $z \in D$. Evidently, all $\Phi_k$ are homeomorphisms.

We consider now inverse mappings $\varphi_k = \Phi_k^{-1} : D_r \rightarrow D$, $k = 0, \ldots, 2n-1$. By construction all of these mappings are embeddings. Moreover, it is easy to see that
$$
u_k(z) = u \circ \varphi_k(z) =
\begin{cases}
	\Re z \,, & \mbox{when } k = 2m \,,\\
	-\Re z \,, & \mbox{when } k = 2m+1 \,.\\
\end{cases}
$$

Let us fix $k \in \{0, \ldots, 2n-1\}$. It is clear that $\varphi_k$ homeomorphically maps a domain
$$
\mathring{D}_r = \{ z \in D \;|\; \Re z > 0 \}
$$
onto a domain
$$
\textstyle
\mathring{R}_k = \Bigl\{ z \in D \bigm| z = a e^{i \varphi},\; a \in (0, 1),\; \varphi \in \Bigl( \frac{\pi(k-1/2)}{2}, \frac{\pi(k+1/2)}{2} \Bigr) \Bigr\} \,,
$$
so with the help of argument similar to the observation preceding to case 1 we conclude that the mapping $\mathring{v}_k = v \circ \varphi_k |_{\mathring{D}_r} : \mathring{D}_r \rightarrow \bR$ is open on level sets of the function $\mathring{u}_k = u \circ \varphi_k |_{\mathring{D}_r} : \mathring{D}_r \rightarrow \bR$. As above, applying Proposition~\ref{prop_monotone} we conclude that function $\mathring{v}_k$ is strictly monotone on each arc
$$
\alpha_c = \mathring{u}_k^{-1}(\mathring{u}_k(c+0i)) = \{ z \in \mathring{D}_r \;|\; \Re z = c \} \,, \quad c \in (0, 1) \,.
$$

We already know that the function $v$ is strictly monotone on the arcs $\gamma_k$ and $\gamma_s$, where $s \equiv k+1 \pmod{2n}$. Therefore the function $v_k = v \circ \varphi_k : D_r \rightarrow \bR$ is strictly monotone on the arcs
\begin{align*}
\alpha_{-} & = \varphi_k^{-1}(\gamma_k) = \{ z \in D_r \;|\; \Re z = 0 \mbox{ and } \Im z < 0 \} \,,\\
\alpha_{+} & = \varphi_k^{-1}(\gamma_s) = \{ z \in D_r \;|\; \Re z = 0 \mbox{ and } \Im z > 0 \} \,.
\end{align*}

Let us verify that $v_k$ is strictly monotone on the arc
$$
\alpha_0 = \alpha_{-} \cup \{0\} \cup \alpha_{+} = u_k^{-1}(u_k(0)) = \{ z \in D_r \;|\; \Re z = 0 \} \,.
$$

Since $v_k(0) = v(0) = V(p) = 0$ according to our initial assumptions and 0 is the boundary point both for $\alpha_{-}$ and $\alpha_{+}$, then $v_k$ is of fixed sign on each of these two arcs.

So we have two possibilities:
\begin{itemize}
	\item either $v_k$ has the same sign on $\alpha_{-}$ and $\alpha_{+}$, then $v_k |_{\alpha_0}$ has a local extremum in $0$;
	\item or $v_k$ has different signs on $\alpha_{-}$ and $\alpha_{+}$, then $v_k$ is strictly monotone on $\alpha_0$.
\end{itemize}

Suppose that $v_k$ has the same sign on $\alpha_{-}$ and $\alpha_{+}$.

We will assume that $v_k$ is negative both on $\alpha_{-}$ and $\alpha_{+}$. The case when $v_k$ is positive on $\alpha_{-}$ and $\alpha_{+}$ is considered similarly.

Denote $z_1 = 0-i/2 \in \alpha_{-}$, $z_2 = 0+i/2 \in \alpha_{+}$. Let
$$
\textstyle
\hat\varepsilon = \frac{1}{2} \min (|v_k(z_1)|, |v_k(z_2)|) > 0 \,.
$$
From the continuity of $v_k$ it follows that there exists $\hat\delta > 0$ to comply with the following implications
\begin{equation}
\begin{aligned}\label{eq_1}
|z-z_1| < \hat\delta & \Rightarrow |v_k(z) - v_k(z_1)| < \hat\varepsilon \,,\\
|z-z_2| < \hat\delta & \Rightarrow |v_k(z) - v_k(z_2)| < \hat\varepsilon \,,\\
|z| = |z-0| < \hat\delta & \Rightarrow |v_k(z) - v_k(0)| = |v_k(z)| < \hat\varepsilon \,.
\end{aligned}
\end{equation}

Let $c \in (0, \hat\delta)$. Then the point $w_0 = c+i0$ is situated on the curve $\alpha_c$ between points $w_1 = c-i/2$ and $w_2 = c+i/2$. It follows from~\eqref{eq_1} that $v_k(w_1) < -\hat\varepsilon$, $v_k(w_2) < -\hat\varepsilon$ and $v_k(w_0) \in (-\hat\varepsilon, 0)$. But these three correlations can not hold true simultaneously since $v_k$ is strictly monotone on $\alpha_c$ as we already know.

The contradiction obtained shows us that $v_k$ has different signs on $\alpha_{-}$ and $\alpha_{+}$. So, $v_k$ is strictly monotone on $\alpha_0$.

Now, repeating argument from case 1 we find such $\varepsilon_k > 0$ and $\delta_k > 0$ that the set
$$
\textstyle
\hat{W}_k = [0, \delta_k) \times \bigl(-\frac{1}{2}, \frac{1}{2}\bigr)
$$
meets the relations
\begin{equation}\label{eq_2}
\begin{alignedat}{2}
	F \circ T \circ \varphi_k(\hat{W}_k) & \supseteq [0, \delta_k) \times (-\varepsilon_k, \varepsilon_k) \,, & \quad \mbox{if } k & = 2m \,,\\
	F \circ T \circ \varphi_k(\hat{W}_k) & \supseteq (-\delta_k, 0] \times (-\varepsilon_k, \varepsilon_k) \,, & \quad \mbox{if } k & = 2m+1 \,.
\end{alignedat}
\end{equation}

Let us denote $W_k = \varphi_k(\hat{W}_k)$,
$$
W = \bigcup_{k=0}^{2n-1} W_k \,, \quad \delta = \min_{k = 0, \ldots, 2n-1} \delta_k > 0 \,, \quad \varepsilon = \min_{k = 0, \ldots, 2n-1} \varepsilon_k > 0 \,.
$$

\begin{figure}[htbp]
\begin{center}
\includegraphics{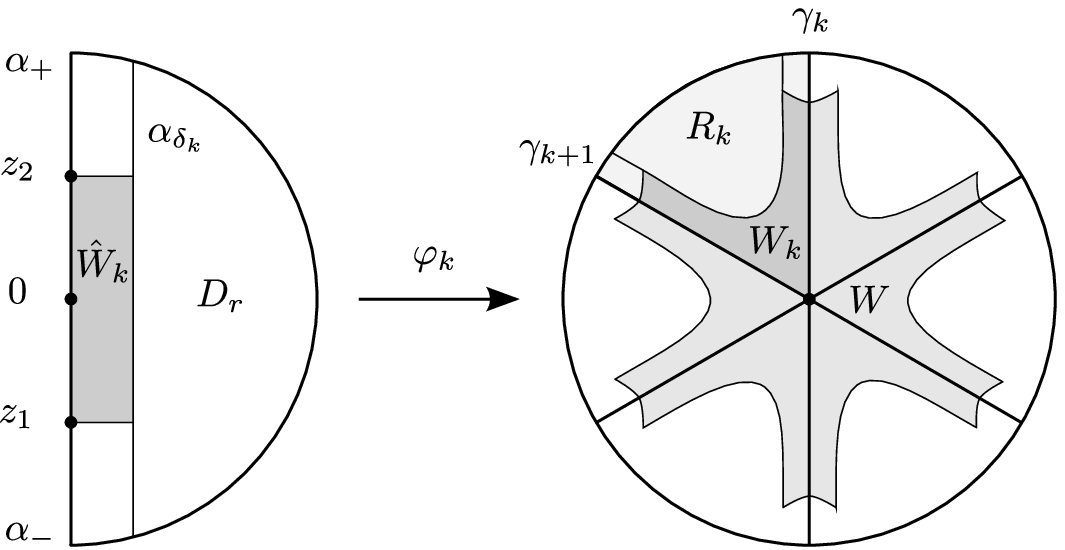}
\caption{}\label{fig_saddle}
\end{center}
\end{figure}

It is easy to show that $W$ is an open neighbourhood of 0 in $D$. From~\eqref{eq_2} and from our initial assumptions it follows that
$$
F(Q) \supseteq F(N) \supseteq F \circ T(W) \supseteq (-\delta, \delta) \times (-\varepsilon, \varepsilon) \,.
$$

\medskip

So, we have proved that for an arbitrary point $p \in M^2$ and its open neighbourhood $Q$ a set $F(Q)$ contains a neigbourhood of $F(p)$. Hence the mapping $F : M^2 \rightarrow \bC$ is open.

At the same time we have shown that an arbitrary $p \in M^2$ is an isolated point of its level set $F^{-1}(F(p))$. It is easy to see now that any level set $F^{-1}(F(p))$ can not contain a nondegenerate continuum.

Consequently, the map $F$ is interior.
\end{proof}

\begin{proof}[Proof of Theorem~\ref{th_conjugace_conditions}]
{\bfseries Necessity.} Let $U, V : M^2 \rightarrow \bR$ be conjugate pseudoharmonic functions on $M^2$ (see Definitions~\ref{defn_conjugate_loc} and~\ref{defn_conjugate_glob}).

Obviously, $V$ is continuous on $M^2$. Suppose that contrary to the statement of Theorem there exists such $c \in \bR$ that $V$ is not open on the level set $\Gamma_c = U^{-1}(c) \subset M^2$, i. e. a map $V_c = V|_{\Gamma_c} : \Gamma_c \rightarrow \bR$ is not open on $\Gamma_c$ in the topology induced from $M^2$.

Let us verify that $V_c$ has a local extremum in some $p \in \Gamma_c$.

Note that the space $\Gamma_c$ is locally arcwise connected, i. e. for every point $a \in \Gamma_c$ and its open neighbourhood $Q$ there exists a neighbourhood $\hat{Q} \subseteq Q$ of $a$ such that every two points $b_1, b_2 \in \hat{Q}$ can be connected by a continuous curve in $Q$. This is a straightforward corollary of the remark subsequent to Definition~\ref{defn_pseudo_harm_loc}.

Since the map $V_c$ is not open by our supposition, then there exists an open subset $O$ of $\Gamma_c$ such that its image $R = V_c(O)$ is not open in $\bR$. Therefore there is a point $d \in R \setminus \Int{R}$. Fix $p \in V_c^{-1}(d) \cap O$.

Let us show that $p$ is a point of local extremum of $V_c$. Fix a neighbourhood $\hat{O} \subseteq O$ of $p$ such that every two points $b_1, b_2 \in \hat{O}$ can be connected by a continuous curve $\beta_{b_1, b_2} : I \rightarrow \Gamma_c$ which meets relations $\beta(0) = b_1$, $\beta(1) = b_2$ and $\beta(I) \subseteq O$. It is clear that an image of a path-connected set under a continuous mapping is path-connected, therefore following inclusions are valid
\begin{eqnarray*}
\bigl(V_c(b_1), V_c(b_2)\bigr) \subset V_c(I) & \mbox{if} & V_c(b_1) < V_c(b_2) \,,\\
\bigl(V_c(b_2), V_c(b_1)\bigr) \subset V_c(I) & \mbox{if} & V_c(b_2) > V_c(b_1) \,.
\end{eqnarray*}
Evidently, $p$ is not an interior point of $V_c(\hat{O})$ since it is not the interior point of $V_c(O)$ by construction and $V_c(\hat{O}) \subseteq V_c(O)$. Then there does not exist a pair of points $b_1, b_2 \in \hat{O}$ such that $V_c(b_1) < V_c(p) < V_c(b_2)$ and either $V(b) \leq V(p)$ for all $b \in \hat{O}$ or $V(b) \geq V(p)$ for all $b \in \hat{O}$, i. e. $p$ is the point of local extremum of $V_c$.

Now, since $V$ is the conjugate pseudo-harmonic function of $U$ in the point $p$ (see Definition~\ref{defn_conjugate_loc}), we can take by definition a neighbourhood $N$ of $p$ in $M^2$ and a homeomorphism $T : D \rightarrow N$ such that a map $f : D \rightarrow \bC$
$$
f(z) = u(z) + i v(z) \,, \quad z \in D
$$
is holomorphic on $D$. Here $u = U \circ T : D \rightarrow \bR$ and $v = V \circ T : D \rightarrow \bR$.

It is clear that without loss of generality we can choose $N$ so small that either $V(b) = V_c(b) \leq V_c(p) = V(p)$ for every $b \in N \cap \Gamma_c$ or $V(b) \geq V(p)$ for all $b \in N \cap \Gamma_c$.

Let for definiteness $p$ is the local maximum of $V_c$ and $V(b) \leq V(p)$ for every $b \in N \cap \Gamma_c$. The case when $p$ is the local minimum of $V_c$ is considered similarly.

On one hand it follows from what we said above that
$$
\bigl(\{U(p)\} \times (V(p), +\infty)\bigr) \cap f(D) = \emptyset
$$
since $u^{-1}(U(p)) = T^{-1}(\Gamma_c \cap N)$ and $v(z) = V(T(z)) \leq V(p)$ for all $z \in T^{-1}(\Gamma_c \cap N)$ by construction. Therefore a point $U(p) + iV(p) = f(T^{-1}(p))$ is not the interior point of a set $f(D)$.

On the other hand it is known that the holomorphic map $f$ is open, so the point $f(T^{-1}(p))$ must be the interior point of the domain $f(D)$.

The contradiction obtained shows that our initial assumption is false and $V$ is open on level sets of $U$.

\medskip

{\bfseries Sufficiency.}
Let $U$ be a pseudo-harmonic function on $M^2$ and a continuous function $V : M^2 \rightarrow \bR$ be open on level sets of $U$.

From Lemma~\ref{lemma_interior_mapping} it follows that the mapping $F : M^2 \rightarrow \bC$, $F(p) = U(p) + iV(p)$, $p \in M^2$ is interior.

Let $p \in M^2$ and $N$ is a simple neighbourhood of $p$ in $M^2$. Then there exists a homeomorphism $T : D \rightarrow N$. It is straightforward that for the open set $N$ a mapping $F_N = F|_N : N \rightarrow \bC$ is interior and its composition $F_N \circ T = F \circ T : D \rightarrow \bC$ with the homeomorfism $T$ is also an interior mapping.

Now from Stoilov theorem it follows that there exists a complex structure on $D$ such that the mapping $F \circ T$ is holomorphic in this complex structure (see~\cite{Stoilov}). But from the uniformization theorem (see~\cite{Forster}) it follows that a simple-connected domain has a unique complex structure. So the mapping $F \circ T$ is holomorphic on $D$ in the standard complex structure. Thus the functions $u = \Re (F \circ T) = U \circ T$ and $v = \Im (F \circ T) = V \circ T$ are conjugate harmonic functions on $D$. Consequently, $V$ is a conjugate pseudo-harmonic function of $U$ in the point $p$.

From arbitrariness in the choise of $p \in M^2$ it follows that $V$ is a conjugate pseudo-harmonic function of $U$ on $M^2$.
\end{proof}

\begin{cor}
Let $U, V : M^2 \rightarrow \bR$ be conjugate pseudoharmonic functions on $M^2$.

Then there exists a complex structure on $M^2$ with respect to which $U$ and $V$ are conjugate harmonic functions on $M^2$.
\end{cor}

\begin{proof}
This statement follows from Theorem~\ref{th_conjugace_conditions}, Lemma~\ref{lemma_interior_mapping} and the Stoilov theorem which says that there exists a complex structure on $M^2$ such that the interior mapping $F(p) = U(p) + iV(p)$, $p \in M^2$ is holomorphic in this complex structure (see~\cite{Stoilov}).
\end{proof}

\end{document}